\documentclass[11pt,reqno]{amsart}
\usepackage[a4paper, hmargin={2.7cm,2.7cm},vmargin={3.3cm,3.3cm}]{geometry}
\usepackage[english]{babel}
\usepackage{color}
\usepackage[noadjust]{cite}
\usepackage{amsmath}
\usepackage{amssymb}
\usepackage{amsfonts}
\usepackage{amsthm}
\usepackage{dsfont}
\usepackage{enumerate}
\usepackage{amsthm}
\usepackage{todonotes}

\usepackage{etoolbox}

\numberwithin{equation}{section}

\makeatletter
\@namedef{subjclassname@2020}{\textup{2020} Mathematics Subject Classification}
\makeatother

\makeatletter
\patchcmd{\ttlh@hang}{\parindent\z@}{\parindent\z@\leavevmode}{}{}
\patchcmd{\ttlh@hang}{\noindent}{}{}{}
\makeatother

\newcommand\numberthis{\addtocounter{equation}{1}\tag{\theequation}}

\theoremstyle{plain}
\newtheorem{theorem}{Theorem}[section]
\newtheorem{lemma}[theorem]{Lemma}
\newtheorem{proposition}[theorem]{Proposition}
\newtheorem{corollary}[theorem]{Corollary}

\theoremstyle{definition}
\newtheorem{definition}[theorem]{Definition}
\newtheorem{examplex}[theorem]{Example}
\newenvironment{example}
  {\pushQED{\qed}\examplex}
  {\popQED\endexamplex}

\theoremstyle{remark}

\DeclareMathOperator{\Rel}{Rel}

\DeclareMathOperator{\loc}{loc}
\DeclareMathOperator*{\esssup}{ess\,sup}

\newcommand{\coeff}{C_{g, \Lambda}}
\newcommand{\frameop}{S_{g, \Lambda}}
\newcommand{\recon}{D_{g, \Lambda}}
\newcommand{\coeffGam}{C_{g, \Gamma}}

\newcommand{\N}{\mathbb{N}}

\newcommand{\Hpi}{\mathcal{H}_{\pi}}

\title[Overcompleteness of coherent frames for unimodular amenable groups]{Overcompleteness of coherent frames \\ for unimodular amenable groups}

\author{Martijn Caspers}

\author{Jordy Timo van Velthoven}

\address{Delft University of Technology,
Mekelweg 4, Building 36,
2628 CD Delft, The Netherlands.}
\email{m.p.t.caspers@tudelft.nl}
\email{j.t.vanvelthoven@tudelft.nl}



\begin{document}

\maketitle

\begin{abstract}
This paper concerns the overcompleteness of coherent frames for unimodular amenable groups. It is shown that for coherent frames associated with a localized vector a set of positive Beurling density can be removed yet still leave a frame.
The obtained results extend various theorems of [J. Fourier Anal. Appl., 12(3):307-344, 2006] to frames with non-Abelian index sets.
\end{abstract}

\section{Introduction}
The aim of this paper is to provide quantitative results on the overcompleteness of a frame in the orbit of a square-integrable representation $(\pi, \Hpi)$ of an amenable unimodular group $G$, i.e., a family of the form
\begin{align} \label{eq:coherent}
 \pi(\Lambda) g = (\pi(\lambda) g)_{\lambda \in \Lambda}
\end{align}
for a vector $g \in \Hpi$ and a discrete $\Lambda \subseteq G$ satisfying the frame inequalities
\[
A \| f \|_{\Hpi}^2 \leq \sum_{\lambda \in \Lambda} |\langle f, \pi(\lambda) g \rangle |^2 \leq B \| f \|_{\Hpi}^2 \quad \text{for all} \quad f \in \Hpi,
\]
for some constants $0 < A \leq B < \infty$. Clearly, any such system is complete in $\Hpi$, i.e., its span is dense in $\Hpi$.
A frame is called \emph{exact} (or a \emph{Riesz basis}) if it ceases to be a frame after the removal of an arbitrary element and is called \emph{overcomplete}, otherwise. The removal of a vector from a frame leaves either a frame or an incomplete system, see, e.g., \cite{duffin1952class}.

For $G = \mathbb{R}^{2d}$ and $\pi$ being the projective Schr\"odinger representation on $L^2 (\mathbb{R}^d)$, the overcompleteness of coherent frames $\pi(\Lambda) g$ (so-called \emph{Gabor frames}) is well-understood through the quantitative framework  \cite{balan2006density,balan2006density2}. Among others, the theory \cite{balan2006density,balan2006density2} provides density conditions for frames and Riesz bases and criteria under which infinite sets can be removed yet still leave a frame.  For possibly non-Abelian groups $G$, density conditions for frames of the form \eqref{eq:coherent} have been obtained more recently in, e.g., \cite{fuehr2017density, enstad2022coherent,enstad2022dynamical, mitkovski2020density, grochenig2020balian}. These density conditions (see also Corollary \ref{cor:density_frames}) assert that
if $\pi(\Lambda) g$ is a frame for $\Hpi$ with an $L^2$-localized vector $g \in \mathcal{B}_{\pi}^2$ (cf. Section \ref{sec:integrable_rep}), then the associated lower Beurling density $D^- (\Lambda)$ of $\Lambda$ satisfies
\begin{align} \label{eq:density_intro}
 D^-(\Lambda) := \lim_{n \to \infty} \inf_{x \in G} \frac{\# (\Lambda \cap xK_n)}{\mu_G (K_n)} \geq d_{\pi},
\end{align}
where $(K_n)_{n \in \mathbb{N}}$ is any strong F\o lner sequence and $d_{\pi} > 0$ the formal degree of $\pi$; see Section \ref{sec:notation}. In addition, necessarily $D^- (\Lambda) = d_{\pi}$ whenever $\pi(\Lambda) g$ is an exact frame.  For (classes of) nilpotent groups $G$, it is also known that for a frame $\pi(\Lambda) g$ with an $L^1$-localized vector $g \in \mathcal{B}_{\pi}^1$, the inequality \eqref{eq:density_intro} must be strict (cf. \cite{grochenig2020balian, groechenig2015deformation, ascensi2014dilation}), so that $\pi(\Lambda) g$ is necessarily overcomplete.

The main result of the present paper provides a criterion for a coherent frame under which a set of positive density can be removed yet leave a frame.

\begin{theorem} \label{thm:intro}
Let $G$ be a second-countable unimodular amenable group with an integrable irreducible projective representation $(\pi, \Hpi)$ of formal degree $d_{\pi} > 0$. Let $\Lambda \subseteq G$ be discrete.

Suppose $\pi(\Lambda) g$ is a frame for $\Hpi$ with $g \in \mathcal{B}_{\pi}^1$ and $D^- (\Lambda) > d_{\pi}$. Then there exists $\Gamma \subseteq \Lambda$ with $D^- (\Gamma) > 0$ such that $(\pi(\lambda) g)_{\lambda \in \Lambda \setminus \Gamma}$ is a frame for $\Hpi$.
\end{theorem}

In addition to Theorem \ref{thm:intro}, the present paper also provides a necessary condition for positive density removal (see Proposition \ref{prop:positive_density_removal}).
Both results extend corresponding theorems of \cite{balan2006density, balan2006density2} to frames arising from possibly non-Abelian groups.
Theorem \ref{thm:intro} applies, in particular, to smooth vectors of square-integrable representations of nilpotent Lie groups (cf. Example \ref{ex:nilpotent}), but also to unimodular groups with possibly exponential growth.
Necessary density conditions for frames arising from nonunimodular groups (e.g., the affine or $ax + b$ group) form currently an open problem.

The possibility of removing sets from an adequate frame yet still leaving a frame can also be deduced from the main results on abstract frames in \cite{freeman2019discretization}; see, e.g., \cite[Corollary 1.5]{freeman2019discretization}. These results do, however, not provide information on the quantity that can be removed, which is the key contribution of Theorem \ref{thm:intro}.

Our proof of Theorem \ref{thm:intro} follows the overall proof structure of the corresponding result for Gabor frames in $L^2 (\mathbb{R}^d)$ (cf. \cite{balan2006density, balan2006density2}). The key ingredients are an identity relating frame measure and Beurling density (Theorem \ref{thm:fundamental_identity}) and a suitable truncation of a Gram matrix (see Lemma \ref{lem:positive_density_removal}). Despite these similarities, there are several important steps that require new methods and techniques in the case of non-Abelian groups. For example, in the setting of general amenable groups, the existence of an adequate ``reference system'' forming a Riesz basis is unknown\footnote{See \cite{groechenig2018orthonormal, oussa2019compactly, oussa2022orthonormal} for constructions of orthonormal bases in the orbit of (classes of) nilpotent Lie groups.} and techniques based on the spectral invariance of matrix algebras are not available in settings with exponential growth; see \cite{fendler2016on, tessera2011inclusion} for examples of settings in which spectral invariance fails. The alternative arguments provided by the present paper to circumvent these obstructions are considered as the main technical contribution and appear to yield more direct proofs even in the case of Abelian index sets.

Lastly, it should be mentioned that for a Gabor frame $\pi(\Lambda) g$ for $L^2 (\mathbb{R}^d)$, in addition to Theorem \ref{thm:intro}, it is possible to choose $\Gamma \subseteq \Lambda$ such that the density of $\Lambda \setminus \Gamma$ is arbitrary close to $d_{\pi} = 1$, see \cite{balan2011redundancy}.
A similar statement for non-Abelian groups remains an open problem.

The paper is organized as follows. Section \ref{sec:notation} provides preliminaries on F\o lner sequences, integrable representations and frames.
 In Section \ref{sec:frame_measure} the notion of a frame measure is introduced and related to formal degree and Beurling density. The main results on overcomplete coherent frames are proven in Section \ref{sec:overcompleteness}.

\section{Notation and preliminary results} \label{sec:notation}
Let $G$ be a second-countable unimodular locally compact group with Haar measure $\mu_G$. Throughout, we fix a compact symmetric unit neighborhood $Q \subseteq G$.

\subsection{F\o lner sequences} \label{sec:foelner}
A \emph{(right) F\o lner sequence} is a sequence $(K_n)_{n \in \mathbb{N}}$ of nonnull compact sets $K_n \subseteq G$ satisfying, for all compact sets $K \subseteq G$,
\[
\lim_{n \to \infty} \frac{\mu_G (K_n K \Delta K_n)}{\mu_G (K_n)} = 0.
\]
The locally compact group $G$ is called \emph{amenable} if it admits a F\o lner sequence. For an amenable group $G$, a F\o lner sequence can be chosen to satisfy the additional properties
\begin{align} \label{eq:nested_folner}
K_n \subseteq K_{n+1} \quad \text{and} \quad G = \bigcup_{n \in \mathbb{N}} K_n,
\end{align}
see, e.g., \cite[Theorem 3.2.1]{emerson1967covering}.

A \emph{(right) strong F\o lner sequence} is a F\o lner sequence $(K_n)_{n \in \N}$ satisfying the stronger condition
\begin{align} \label{eq:strong_folner}
\lim_{n \to \infty} \frac{\mu_G ( K_n K \cap K_n^c K)}{\mu_G (K_n)} = 0
\end{align}
for all compact sets $K \subseteq G$. If $(K_n)_{n \in \mathbb{N}}$ is a F\o lner sequence and $U \subseteq G$ is a compact symmetric unit neighborhood, then $(K_n U)_{n \in \mathbb{N}}$ is a strong F\o lner sequence (cf. \cite[Proposition 5.10]{pogorzelski2021leptin}).
Clearly, also strong F\o lner sequences exist with the additional properties \eqref{eq:nested_folner}.

\subsection{Discrete sets}
A set $\Lambda \subseteq G$ is called \emph{relatively separated} if, for some (all) compact unit neighborhoods $U \subseteq G$,
\[
 \sup_{x \in G} \# (\Lambda \cap x U) < \infty.
\]
For a relatively separated $\Lambda$, its relative separation (relative to the fixed neighborhood $Q$) is defined to be $\Rel(\Lambda) := \sup_{x \in G} \# (\Lambda \cap x Q) < \infty$.
Given a compact unit neighborhood $U$, a set $\Lambda \subseteq G$ is called \emph{$U$-dense} if $G = \bigcup_{\lambda \in \Lambda} \lambda U$. Equivalently, $\Lambda$ is $U$-dense if $\# (\Lambda \cap x U) \geq 1$ for all $x \in G$.
A set is \emph{relatively dense} if it is $U$-dense for some compact unit neighborhood $U$.

\subsection{Local maximal functions}
For $F \in L^{\infty}_{\loc} (G)$, its (left-sided) local maximal function $M^L F : G \to [0, \infty)$ is defined by
\[
M^L F (x) = \esssup_{z \in Q} |F(x z)|, \quad x \in G.
\]
The associated (left-sided) Wiener amalgam space $W^L (L^p)$, with $p \in [1, 2]$, is defined by
\[
W^L (L^p):= \big\{ F \in L^{\infty}_{\loc} (G) : M^L F \in L^p (G) \big\}.
\]
Each space $W^L (L^p)$, $p \in [1, 2]$,  satisfies $W^L (L^p) \hookrightarrow L^p$, and additionally $W^L (L^p) \hookrightarrow L^{\infty}$.

The following restriction property will be essential in the sequel, see, e.g., \cite[Lemma 1]{grochenig2008homogeneous}.

\begin{lemma} \label{lem:restriction}
Let $\Lambda \subseteq G$ be relatively separated and let $F \in W^L(L^2)$ be continuous. For any compact set $K \subseteq G$,
\[
 \sum_{\lambda \in \Lambda \cap K^c} | F(\lambda)|^2 \leq \frac{\Rel (\Lambda)}{\mu_G (Q)} \int_{K^c Q} |M^L F(x)|^2 \; d\mu_G (x).
\]
In particular, for every $\varepsilon > 0$, there exists compact $K \subseteq G$ such that $\sum_{\lambda \in \Lambda \cap K^c} |F(\lambda) |^2 \leq \varepsilon$.
\end{lemma}

In a similar fashion as above, the right-sided local maximal function $M^R L$ of $F \in L^{\infty}_{\loc} (G)$ is defined by $M^R F (x) = \esssup_{z \in Q} |F(zx)|$. The associated (two-sided) Wiener amalgam space $W(L^1)$ is defined as
\[
W(L^1) := \big\{ F \in L^{\infty}_{\loc} (G) : M^L M^R F \in L^1 (G) \big\}
\]
and equipped with the norm $\| F \|_W := \| M^L M^R F \|_{L^1}$.

The local maximal functions satisfy
\begin{align} \label{eq:maximal_convolution}
M^L (F_1 \ast F_2) \leq |F_1| \ast M^L F_2 \quad \text{and} \quad M^R (F_1 \ast F_2) \leq M^R F_1 \ast |F_2|,
\end{align}
provided the convolution product $F_1 \ast F_2$ is (almost everywhere) well-defined. In particular, the inequalities \eqref{eq:maximal_convolution} imply that $(W^L (L^1))^{\vee} \ast W^L(L^1) \hookrightarrow W(L^1)$, where the involution $^\vee$ is defined as $F^{\vee} (x) = F(x^{-1})$ for $x \in G$.

\subsection{Integrable representations} \label{sec:integrable_rep}
A projective unitary representation $(\pi, \Hpi)$ on a Hilbert space $\Hpi$ is a strongly measurable map $\pi : G \to \mathcal{U}(\Hpi)$ satisfying
\[
\pi(xy) = \sigma(x,y) \pi(x) \pi(y), \quad x,y \in G,
\]
for a function $\sigma : G \times G \to \mathbb{T}$.  For a vector $g \in \Hpi$, the associated coefficient transform $V_g : \Hpi \to L^{\infty} (G)$ is defined through the matrix coefficients
\[
V_g f (x) = \langle f, \pi(x) g \rangle, \quad x \in G.
\]
The absolute value $|V_g f| : G \to [0, \infty)$ is continuous for all $f, g \in \Hpi$; see \cite[Theorem 7.5]{varadarajan1985geometry}.

A projective representation $(\pi, \Hpi)$ is said to be \emph{irreducible} if $\{0\}$ and $\Hpi$ are the only closed subspaces of $\Hpi$ invariant under all operators $\pi(x)$ for $x \in G$.

An irreducible projective representation $(\pi, \Hpi)$ is called \emph{square-integrable} or a \emph{discrete series representation} of $G$ if there exists $g \in \Hpi \setminus \{0\}$ such that
\begin{align*}
\int_G |V_g g (x) |^2 \; d \mu_G (x) = \int_G | \langle g, \pi(x) g \rangle |^2 \; d\mu_G (x) < \infty.
\end{align*}
The significance of a discrete series representation $(\pi, \Hpi)$ of $G$ is that there exists $d_{\pi} > 0$, called the \emph{formal degree} of $\pi$, such that
\begin{align} \label{eq:ortho}
\int_G \langle f_1, \pi(x) g_1 \rangle \langle \pi(x) g_2, f_2 \rangle \; d\mu_G (x) = d_{\pi}^{-1} \langle f_1, f_2 \rangle \overline{\langle g_1, g_2 \rangle}
\end{align}
for all $f_1, f_2, g_1, g_2 \in \Hpi$. For a square-integrable representation $\pi$, we define the subspace
\[
\mathcal{B}^2_{\pi} := \big\{ g \in \Hpi : V_g g \in W^L (L^2) \big\}.
\]
Then $\mathcal{B}^2_{\pi}$ is nonzero and norm dense in $\Hpi$, see, e.g., \cite{fuehr2007sampling, grochenig2008homogeneous}.

In addition to square-integrability, a vector $g \in \Hpi \setminus \{0\}$ satisfying
$
\int_G | V_g g (x) | \; d\mu_G (x) < \infty$ is called an \emph{integrable vector}. An \emph{integrable representation} is an irreducible representation admitting an integrable vector.
For an integrable representation $\pi$, we also consider the subspace
\[
\mathcal{B}^1_{\pi} := \big\{ g \in \Hpi : V_g g \in W^L (L^1) \big\}.
\]
If $g \in \Hpi$ is an integrable vector and $h \in C_c (G) \setminus \{0\}$, then the associated G\aa rding vector $\pi(h) g := \int_G h(x) \pi(x) g \; d\mu_G (x)$ defines an element of $\mathcal{B}^1_{\pi}$. Therefore, the space $\mathcal{B}^1_{\pi}$ is nonzero and norm dense in $\Hpi$.

The following simple lemma will be used below.

\begin{lemma} \label{lem:bettervectors}
Let $(\pi, \Hpi)$ be an irreducible integrable representation of $G$. Then $\mathcal{B}^1_{\pi} \subseteq \mathcal{B}^2_{\pi}$ and
$
\mathcal{B}^1_{\pi} = \{ g \in \Hpi : V_g g \in W(L^1) \}.
$
\end{lemma}
\begin{proof}
Let $g \in \mathcal{B}^1_{\pi}$ be nonzero. The orthogonality relations \eqref{eq:ortho} yield
\[
V_g g(x) = d_{\pi} \| g \|^{-2}_{\Hpi} \int_G \langle g , \pi(y) g \rangle \langle \pi(y) g, \pi(x) g \rangle \; d\mu_G (y), \quad x \in G.
\]
Set $C := d_{\pi} \| g \|^{-2}_{\Hpi}$. Then
$
|V_g g | (x) \leq C (|V_g g| \ast |V_g g|)(x)$ for all $x \in G$. By Equation \eqref{eq:maximal_convolution}, it follows therefore that
$ M^L V_g g \leq C (|V_g g| \ast M^L V_g g), $
and thus $L^2 (G) \ast L^1 (G) \hookrightarrow L^2 (G)$ implies that \[ \| M^L V_g g \|_{L^2} \leq C \| V_g g \|_{L^2} \| M^L V_g g \|_{L^1}.\] Similarly, it follows that $M^L M^R V_g g \leq C (M^R V_g g \ast M^L V_g g)$. Since $|V_g g|^{\vee} = |V_g g|$, and hence $M^R V_g g = (M^L V_g g)^{\vee}$, this implies
$\| V_g g \|_{W} \leq C \| M^L V_g g \|_{L^1} \| M^L V_g g\|_{L^1}$.
\end{proof}

Lastly, we mention a class of groups and projective representations for which $\mathcal{B}^1_{\pi}$ is nonzero.

\begin{example} \label{ex:nilpotent}
Let $N$ be a connected, simply connected nilpotent Lie group and let $(\pi, \Hpi)$ be an irreducible unitary representation of $N$. Denote by $\Hpi^{\infty}$ the (dense) subspace of smooth vectors of $\pi$, i.e.,
the space of all vectors $g \in \Hpi$ such that the orbit map $x \mapsto \pi(x) g$ is smooth.

Suppose that $\pi$ is square-integrable modulo the center $Z$ of $N$, meaning that there exists nonzero $g \in \Hpi$ such that
\[
\int_{N/Z} |\langle g, \pi(x) g \rangle |^2 \; d\mu_{N/Z} (x) < \infty.
\]
Then, given a smooth cross-section $s : N/Z \to N$, the mapping $\pi' := \rho \circ s$ forms a (projective) discrete series representation of $G := N/Z$.
Moreover, for any smooth vector $g \in \Hpi^{\infty}$, the function $V_g g  = \langle g , \pi' (\cdot) g \rangle$ is a Schwartz function on $G$ (see \cite[Theorem 4.5.11]{corwin1990representations}), and hence $g \in \mathcal{B}^1_{\pi}$. See, e.g., \cite[Section 6.2]{bedos2022smooth} for further details and properties.

The interested reader is referred to \cite{nielsen1983unitary} for a list of low-dimensional nilpotent Lie groups and explicit realizations of their irreducible repesentations in $L^2 (\mathbb{R}^d)$ for some suitable $d \in \mathbb{N}$.
\end{example}

\subsection{Coherent frames}
Let $(\pi, \Hpi)$ be a square-integrable projective representation of $G$. For a nonzero vector $g \in \Hpi$ and a discrete set $\Lambda \subseteq G$, a family
$
\pi(\Lambda) g = \big ( \pi(\lambda) g )_{ \lambda \in \Lambda }
$
is called a \emph{coherent system} in $\Hpi$. A coherent system $\pi(\Lambda) g$ is called a \emph{frame} for $\Hpi$ if there exist $A, B > 0$, called \emph{frame bounds}, such that
\[
A \| f \|_{\Hpi}^2 \leq \sum_{\lambda \in \Lambda} |\langle f, \pi(\lambda) g \rangle |^2 \leq B \| f \|_{\Hpi}^2, \quad f \in \Hpi.
\]
Equivalently, the system $\pi(\Lambda) g$ is a frame if the \emph{frame operator}
\[
\frameop : \Hpi \to \Hpi, \quad f \mapsto \sum_{\lambda \in \Lambda} \langle f, \pi(\lambda) g \rangle \pi(\lambda) g
\]
is bounded and invertible. If $\pi(\Lambda) g$ is a frame for $\Hpi$ with frame bounds $A$ and $B$, then the system $(h_{\lambda} )_{\lambda \in \Lambda}$ given by $h_{\lambda} := \frameop^{-1} \pi(\lambda) g$ is a frame for $\Hpi$ with frame bounds $B^{-1}$ and $A^{-1}$, called the \emph{canonical dual frame} of $\pi(\Lambda) g$.
The systems $\pi(\Lambda) g$ and $(h_{\lambda} )_{\lambda \in \Lambda}$ satisfy
$
0 < \langle \pi(\lambda) g, h_{\lambda} \rangle \leq 1
$
for all $\lambda \in \Lambda$. A frame for which the frame bounds can be chosen to be  $A=B=1$ is called a \emph{Parseval frame}.

The following well-known covering properties of the index set of a coherent frame will be used below, see, e.g., \cite{groechenig2015deformation, fuehr2017density, enstad2022dynamical} for proofs.

\begin{lemma}
If $\pi(\Lambda) g$ is a frame for $\Hpi$ with $g \in \mathcal{B}_{\pi}^2$, then $\Lambda$ is relatively separated and relatively dense.
\end{lemma}

For a frame $\pi(\Lambda) g$ for $\Hpi$, the associated \emph{coefficient operator} $\coeff : \Hpi \to \ell^2 (\Lambda)$ is defined by
$
f \mapsto (\langle f, \pi(\lambda) g \rangle )_{\lambda \in \Lambda}.
$
Its adjoint $\recon := \coeff^*$ is the \emph{reconstruction operator}, given by $\recon c = \sum_{\lambda \in \Lambda} c_{\lambda} \pi(\lambda) g$ for $c \in \ell^2 (\Lambda)$. The \emph{Gramian operator} is the composition $\coeff \recon$ on $\ell^2 (\Lambda)$, which will be identified with the matrix $(\langle \pi(\lambda) g, \pi(\lambda') g \rangle)_{\lambda, \lambda' \in \Lambda}$.

\section{Frame measure and Beurling density} \label{sec:frame_measure}
Henceforth, let $(\pi, \mathcal{H}_{\pi})$ be a discrete series representation of $G$ of formal dimension $d_{\pi} > 0$.

\subsection{Frame measure}
In this section we define a notion of frame measure for a given coherent frame. This notion is a special case of the so-called \emph{ultrafilter frame measure function} for abstract frames as considered in \cite{balan2007measure}.

\begin{definition} \label{def:frame_measure}
Let $(K_n)_{n \in \mathbb{N}}$ be a strong F\o lner sequence in $G$ satisfying the cover property \eqref{eq:nested_folner}. Let $\pi(\Lambda) g$ be a coherent frame for $\Hpi$ with canonical dual frame $(h_{\lambda})_{\lambda \in \Lambda}$.

The \emph{lower} and \emph{upper frame measure} of $\pi(\Lambda) g$ are defined by
\begin{align*}
    M^- (\mathcal{G}_{\Lambda}) := \lim_{n \to \infty} \inf_{x \in G} \frac{1}{\#( \Lambda \cap x K_n)} \sum_{\lambda \in \Lambda \cap x K_n} \langle \pi(\lambda) g, h_{\lambda} \rangle
\end{align*}
and
\begin{align*}
    M^+ (\mathcal{G}_{\Lambda}) := \lim_{n \to \infty} \sup_{x \in G} \frac{1}{\#( \Lambda \cap x K_n)} \sum_{\lambda \in \Lambda \cap x K_n} \langle \pi(\lambda) g, h_{\lambda} \rangle,
\end{align*}
respectively.
\end{definition}

It will follow from Theorem \ref{thm:fundamental_identity} (cf. Corollary \ref{cor:frame_measure_independent}) that the frame measures of a coherent frame $\pi(\Lambda) g$ with $g \in \mathcal{B}_{\pi}^2$ are independent of the choice of strong F\o lner sequence.

\subsection{Beurling density}
For a discrete set $\Lambda \subseteq G$, its \emph{lower} and \emph{upper Beurling density} are defined by
\begin{align*} \label{eq:lower_density}
    D^- (\Lambda) := \lim_{n \to \infty} \inf_{x \in G} \frac{\#( \Lambda \cap x K_n)}{\mu_G (K_n)}
\quad \text{resp.} \quad
    D^+ (\Lambda) := \lim_{n \to \infty} \sup_{x \in G} \frac{\#( \Lambda \cap x K_n)}{\mu_G (K_n)} ,
\end{align*}
where $(K_n)_{n \in \mathbb{N}}$ is any strong F\o lner sequence. The definition of $D^-$ and $D^+$ are independent of the choice of F\o lner sequence, cf. \cite[Proposition 5.14]{pogorzelski2021leptin}.

The following theorem relates the notions of frame measure and Beurling density. In particular, it shows that the frame measures only depend on the density of the index set and the formal degree of the representation.

\begin{theorem} \label{thm:fundamental_identity}
Suppose $\pi(\Lambda) g$ is a frame for $\mathcal{H}_{\pi}$ with $g \in \mathcal{B}^2_{\pi}$. Then
\[ M^-(\mathcal{G}_{\Lambda}) = \frac{d_{\pi}}{D^+ (\Lambda)} \quad \text{and} \quad M^+ (\mathcal{G}_{\Lambda}) = \frac{d_{\pi}}{D^- (\Lambda)}.
\]
\end{theorem}

\begin{proof}
Without loss of generality, it will be assumed throughout the proof that $\| g \|_{\Hpi} = d_{\pi}^{1/2}$, so that $V_g : \Hpi \to L^2 (G)$ is an isometry. Write $g_{\lambda} = \pi(\lambda) g$ for $\lambda \in \Lambda$. Suppose $(g_{\lambda})_{\lambda \in \Lambda}$ is frame for $\Hpi$ with frame bounds $A, B >0$. Then the index set $\Lambda$ is relatively separated and relatively dense.  Therefore, there exists $n_0 \in \mathbb{N}$ such that
$
1 \leq \# (\Lambda \cap x K_{n_0} ) < \infty$ for all $x \in G$.

Let $\varepsilon > 0$ and $x \in G$ be arbitrary and let $n \in \mathbb{N}$ be such that $n \geq n_0$. In addition, fix a symmetric compact unit neighorhood $K \subseteq G$ such that
\begin{align} \label{eq:choice_K}
 \int_{G/K} |V_ g g (y) |^2 \; d\mu_G (y) \leq \varepsilon^2
 \quad \text{and} \quad
 \sum_{\lambda \in \Lambda \cap K^c} | V_g g (\lambda) |^2 \leq \varepsilon^2,
\end{align}
cf. Lemma \ref{lem:restriction} and Equation \eqref{eq:ortho}.
For fixed $y \in G$, it follows that
\begin{align*}
    d_{\pi}
    &= \| \pi(y) g \|^2_{\Hpi}
    = \bigg \langle \sum_{\lambda \in \Lambda} \langle \pi(y) g , h_{\lambda} \rangle g_{\lambda} , \pi(y) g \bigg \rangle
   = \sum_{\lambda \in \Lambda} V_{g} g_{\lambda} (y) \overline{V_{g} h_{\lambda} (y)}.
\end{align*}
Define $H(y) := \sum_{\lambda \in \Lambda} V_{g} g_{\lambda} (y) \overline{V_{g} h_{\lambda} (y)}$ for $y \in G$, and write
$H(y) = \sum_{i = 1}^3 H_i (y)$, where
\begin{align*}
    H_1 (y) = & \sum_{\lambda \in \Lambda \cap x(K_n \setminus K_n^c K)} V_{g} g_{\lambda} (y) \overline{V_{g} h_{\lambda} (y)}, \quad H_2 (y) = \sum_{\lambda \in \Lambda \cap x(K_n K)^c} V_{g} g_{\lambda} (y) \overline{V_{g} h_{\lambda} (y)},
\end{align*}
and
\[
 H_3 (y) = \sum_{\lambda \in \Lambda \cap (x K_n K \setminus x ( K_n \setminus  K^c_n K))}  V_{g} g_{\lambda} (y) \overline{V_{g} h_{\lambda} (y)}
= \sum_{\lambda \in \Lambda \cap x K_n K \cap x K^c_n K}  V_{g} g_{\lambda} (y) \overline{V_{g} h_{\lambda} (y)}.
\]

The proof is split into four steps.
\\\\
\textbf{Step 1.} This step provides estimates of $T_i := \int_{x K_n} H_i (y) \; d\mu_G (y)$ for $i = 1,2,3$. Similar estimates for metric balls in settings with polynomial growth can be found in \cite{fuehr2017density, mitkovski2020density, nitzan2020revisiting}.
\\\\
\emph{Estimate $T_1$.} Note that a direct calculation entails
\[
T_1 = \int_{G} H_1 (y) \; d\mu_G (y) - \int_{G \setminus x K_n} H_1 (y) \; d\mu_G (y) = \sum_{\lambda \in \Lambda \cap x (K_n \setminus K_n^c K)} \big \langle V_{g} g_{\lambda}, V_{g} h_{\lambda} \rangle_{L^2} - L,
\]
where $L := \int_{G \setminus x K_n}  H_1 (y) \; d\mu_G (y) = \sum_{\lambda \in \Lambda \cap x (K_n \setminus K_n^c K)} \int_{G \setminus x K_n} V_{g} g_{\lambda} (y) \overline{V_{g} h_{\lambda} (y)} \; d\mu_G (y)$.
For estimating $L$, note first that an application of Cauchy-Schwarz' inequality gives
\[
\bigg| \int_{G \setminus xK_n} V_{g} g_{\lambda} (y) \overline{V_{g} h_{\lambda} (y)} d\mu_G (y) \bigg| \leq \bigg( \int_{G \setminus xK_n} |V_{g} g_{\lambda} (y) |^2 \; d\mu_G (y) \bigg)^{1/2} \| h_{\lambda} \|_{\Hpi},
\]
where it is used that $\| V_{g} h_{\lambda} \|_{L^2} = \| h_{\lambda} \|_{\Hpi}$. Since $\lambda \in x   (K_n \setminus K_n^c K)$, it follows that $\lambda K \subseteq x K_n$. Hence, a change-of-variable gives
\begin{align*}
 \int_{G \setminus x K_n} |V_{g} g (\lambda^{-1} y) |^2 \; d\mu_G (y)
 \leq \int_{G \setminus \lambda K} |V_{g} g (\lambda^{-1} y) |^2 \; d\mu_G (y)
 = \int_{G \setminus K} |V_{g} g (y) |^2 \; d\mu_G (y).
\end{align*}
Therefore, Equation \eqref{eq:choice_K} yields
\[
\bigg( \int_{G / x K_n} |V_{g} g_{\lambda} (y) |^2 \; d\mu_G (y) \bigg)^{1/2} \leq \varepsilon.
\]
Hence,
\[
|L| \leq \varepsilon \sum_{\lambda \in \Lambda \cap x (K_n \setminus K_n^c K) } \| h_{\lambda} \|_{\Hpi} \leq \varepsilon A^{-1/2} \# (\Lambda \cap x K_n),
\]
where it used that $\| h_{\lambda} \|_{\Hpi} \leq A^{-1/2}$ for all $\lambda \in \Lambda$.
\\\\
\emph{Estimate $T_2$.} An application of Cauchy-Schwarz' inequality gives
\begin{align*}
    \bigg| \int_{x K_n} H_2 (y) \; d\mu_G (y) \bigg| \leq \int_{x K_n} \bigg( \sum_{\lambda \in \Lambda \cap x (K_n K)^c} |V_{g} g_{\lambda} (y)|^2 \bigg)^{\frac{1}{2}} \bigg( \sum_{\lambda \in \Lambda} |V_{g} h_{\lambda} (y) |^2 \bigg)^{\frac{1}{2}} \; d\mu_G (y).
\end{align*}
For $y \in x K_n$ and $\lambda \in x (K_n K)^c$, one has $\lambda \notin x K_n K$ and hence $\lambda \notin y K$. Therefore,
\[
\bigg( \sum_{\lambda \in \Lambda \cap x (K_n K)^c} |V_{g} g_{\lambda} (y)|^2 \bigg)^{\frac{1}{2}}
\leq \bigg( \sum_{\lambda \in \Lambda \cap  yK^c} |V_{g} g_{\lambda} (y)|^2 \bigg)^{\frac{1}{2}} = \bigg( \sum_{\lambda \in \Lambda \cap y K^c} |V_g g (y^{-1} \lambda)|^2 \bigg)^{\frac{1}{2}}.
\]
By Equation \eqref{eq:choice_K}, it holds that
\[
\bigg( \sum_{\lambda \in \Lambda \cap K^c} |V_g g ( \lambda)|^2 \bigg)^{\frac{1}{2}} \leq \varepsilon,
\]
and hence
\begin{align*}
\bigg| \int_{x K_n} H_2 (y) \; d\mu_G (y) \bigg| &\leq \varepsilon  \int_{x K_n} \bigg( \sum_{\lambda \in \Lambda} |\langle \pi(y) g, h_{\lambda} \rangle |^2 \bigg)^{\frac{1}{2}} \; d\mu_G (y) \leq \varepsilon A^{-1/2} \| g \|_{\Hpi} \mu_G (x K_n) \\
&= \varepsilon A^{-1/2} d_{\pi}^{1/2} \mu_G ( K_n)
\end{align*}
by the frame property of $(h_{\lambda})_{\lambda \in \Lambda}$.
\\\\
\emph{Estimate $T_3$.} A direct calculation gives
\begin{align*}
\int_{x K_n} |H_3 (y)| \; d\mu_G (y)
&\leq \sum_{\lambda \in \Lambda \cap x K_n K \cap x K_n^c K} \int_G |V_{g} g_{\lambda} (y)| |V_{g} h_{\lambda} (y)| \; d\mu_G (y) \\
&\leq \sum_{\lambda \in \Lambda \cap x K_n K \cap x K_n^c K} \| V_{g} g_{\lambda} \|_{L^2} \| V_{g} h_{\lambda} \|_{L^2} \\
&\leq A^{-1/2}  B^{1/2} \# \big(\Lambda \cap  (xK_n K \cap  xK_n^c K)\big),
\end{align*}
where it is used that $\| g_{\lambda} \|_{\Hpi} \leq B^{1/2}$ and $ \| h_{\lambda} \|_{\Hpi} \leq A^{-1/2}$ for all $\lambda \in \Lambda$.
\\\\
\textbf{Step 2.} Using the notation of Step 1, we have
\[
\sum_{\lambda \in \Lambda \cap x (K_n \setminus K_n^c K)} \langle g_{\lambda}, h_{\lambda} \rangle = \int_{x K_n} H(y) \; d\mu_G (y) - T_2 - T_3 + L.
\]
This implies that
\begin{align*}
 & \bigg | \int_{x K_n} H(y) \; d\mu_G (y) - \sum_{\lambda \in \Lambda \cap x K_n} \langle g_{\lambda}, h_{\lambda} \rangle \bigg |   \\
 &  \quad \quad \quad =    \bigg | \int_{x K_n} H(y) \; d\mu_G (y) - \sum_{\lambda \in \Lambda \cap x (K_n \setminus K_n^c K)} \langle g_{\lambda}, h_{\lambda} \rangle - \sum_{\lambda \in \Lambda \cap x K_n \setminus  (x (K_n \setminus K_n^c K)) } \langle g_{\lambda}, h_{\lambda} \rangle \bigg |         \\
    & \quad \quad \leq  |T_2| + |T_3| + |L| + \sum_{\lambda \in \Lambda \cap xK_n \cap xK_n^c K} | \langle g_{\lambda}, h_{\lambda} \rangle | \\
    &\quad \quad \leq \varepsilon A^{-1/2} d_{\pi}^{1/2} \mu_G (K_n) + A^{-1/2} B^{1/2} \# \big(\Lambda \cap  (xK_n K \cap  xK_n^c K)\big) \\
    & \quad \quad \quad \quad  + \varepsilon A^{-1/2} \# (\Lambda \cap x K_n) + \# (\Lambda \cap (xK_n \cap x K_n^c K)), \numberthis \label{eq:cardinality}
\end{align*}
where the last step used the estimates of $T_2, T_3$ and $L$ (cf. Step 2) together with $|\langle g_{\lambda}, h_{\lambda} \rangle | \leq 1$.

To further estimate the difference \eqref{eq:cardinality}, we use a suitable upper bound for the cardinality $\# (\Lambda \cap (xK_n K \cap x K_n^c K))$. By a standard packing argument, there exists $C(K, \Lambda) > 0$ such that, for all compact sets $U \subseteq G$,
\[
\# (\Lambda \cap U) \leq C \mu_G (UK),
\]
see, e.g., \cite[Lemma 2.4]{pogorzelski2021leptin} or \cite[Corollary 3.4]{enstad2022dynamical}. Applying this to the sets $U = x K_n K \cap x K_n^c K$ yields that
\begin{align*}
\# (\Lambda \cap (x K_n K \cap x K_n^c K)) &\leq C \mu_G ((x K_n K \cap x_n K_n^c K)K) =
C \mu_G  (xK_n K^2 \cap x K_n^c K^2 ) \\
&= C \mu_G  (K_n K^2 \cap  K_n^c K^2 ). \numberthis \label{eq:cardinality_measure}
\end{align*}
Setting $C' := (1+ A^{-1/2} B^{1/2}) C$, it follows therefore from combining \eqref{eq:cardinality} and \eqref{eq:cardinality_measure} that
\begin{align*}
    & \bigg |\int_{x K_n} H(y) \; d\mu_G (y) -  \sum_{\lambda \in \Lambda \cap x K_n} \langle g_{\lambda}, h_{\lambda} \rangle \bigg |  \\
    &\quad \quad \leq \varepsilon A^{-1/2} d_{\pi}^{1/2} \mu_G (K_n) + \varepsilon A^{-1/2} \# (\Lambda \cap x K_n) + C' \mu_G  (K_n K^2 \cap  K_n^c K^2 ).
    \end{align*}
    with all constants independent of $x$ and $n$.
    \\\\
    \textbf{Step 3.}
    Recall that $\mu_G (K_n)^{-1} \int_{x K_n} H(y) \; d\mu_G (y) = d_{\pi}$.
    Therefore, the estimates obtained in Step 2 imply that
\begin{align*}
& \bigg| d_{\pi}  - \frac{1}{\mu_G(K_n)} \sum_{\lambda \in \Lambda \cap x K_n} \langle g_{\lambda}, h_{\lambda} \rangle \bigg|  \\
&\quad \quad \leq \varepsilon A^{-1/2} d_{\pi}^{1/2} + \varepsilon A^{-1/2} \frac{\# (\Lambda \cap x K_n)}{\mu_G (K_n)} + C' \frac{\mu_G (K_n K^2 \cap  K_n^c K^2 )}{\mu_G (K_n)}.
\end{align*}
Multiplying both sides with $\mu_G(K_n) / \# (\Lambda \cap xK_n)$
yields
\begin{align*}
    & \bigg| d_{\pi} \bigg( \frac{\# (\Lambda \cap x K_n)}{\mu_G(K_n)} \bigg)^{-1} -
    \frac{1}{\# (\Lambda \cap x K_n)} \sum_{\lambda \in \Lambda \cap x K_n} \langle g_{\lambda}, h_{\lambda} \rangle \bigg| \\
    &\quad \quad \leq \varepsilon A^{-1/2}  + \varepsilon A^{-1/2} d_{\pi}^{1/2} \bigg( \frac{\# (\Lambda \cap x K_n)}{\mu_G(K_n)} \bigg)^{-1} \\
    &\quad \quad + C'   \frac{\mu_G (K_n K^2 \cap  K_n^c K^2 )}{\mu_G (K_n)} \bigg( \frac{\# (\Lambda \cap x K_n)}{\mu_G(K_n)} \bigg)^{-1} . \numberthis \label{eq:final_estimate}
\end{align*}
By the strong F\o lner property \eqref{eq:strong_folner}, it follows that
\[
\lim_{n \to \infty}  \frac{\mu_G (K_n K^2 \cap  K_n^c K^2 )}{\mu_G (K_n)} = 0.
\]
 Therefore,
\begin{equation}\label{Eqn=GoesToZero}
\lim_{n \rightarrow \infty} \sup_{x \in G} \bigg| d_{\pi} \bigg( \frac{\# (\Lambda \cap x K_n)}{\mu_G(K_n)} \bigg)^{-1} -
    \frac{1}{\# (\Lambda \cap x K_n)} \sum_{\lambda \in \Lambda \cap x K_n} \langle g_{\lambda}, h_{\lambda} \rangle \bigg| = 0,
\end{equation}
where it is used that $D^- (\Lambda) > 0$ since $\Lambda$ is relatively dense, see, e.g., \cite[Lemma 3.8]{pogorzelski2021leptin}.
\\~\\
\textbf{Step 4.} Using \eqref{Eqn=GoesToZero}, the conclusion $\frac{d_{\pi}}{D^{-} (\Lambda)} = M^+(\mathcal{G}_{\Lambda})$ can be shown as follows. For $i \in \mathbb{N}$, choose $n_i \in \mathbb{N}$ increasing and $x_i \in G$  such that
\[
 M^+ (\mathcal{G}_{\Lambda}) =   \lim_{i \rightarrow \infty}  \frac{1}{\# (\Lambda \cap x_i K_{n_i})} \sum_{\lambda \in \Lambda \cap x_i K_{n_i}} \langle g_{\lambda}, h_{\lambda} \rangle.
 \]
Then, by \eqref{Eqn=GoesToZero} and definition of the lower Beurling density,
\[
 M^+ (\mathcal{G}_{\Lambda})
 =    \lim_{i \rightarrow \infty} d_{\pi} \bigg( \frac{\# (\Lambda \cap x_i K_{n_i})}{\mu_G(K_{n_i})} \bigg)^{-1}
 \leq   \frac{d_{\pi}}{D^{-} (\Lambda) }.
\]
Conversely, for $i \in \mathbb{N}$  choose $n_i \in \mathbb{N}$ increasing and $x_i \in G$  such that
\[
D^{-} (\Lambda) = \lim_{i \to \infty} \frac{\# (\Lambda \cap x_i K_{n_i})}{\mu_G(K_{n_i})}.
\]
Then, by \eqref{Eqn=GoesToZero} and definition of $M^+ (\mathcal{G}_{\Lambda})$,
\[
 \frac{d_{\pi}}{D^{-} (\Lambda) } =  \lim_{i \rightarrow \infty}  \frac{1}{\# (\Lambda \cap x_i K_{n_i})} \sum_{\lambda \in \Lambda \cap x_i K_{n_i}} \langle g_{\lambda}, h_{\lambda} \rangle
 \leq  M^+ (\mathcal{G}_{\Lambda}).
\]
The identity $\frac{d_{\pi}}{D^{+} (\Lambda)} = M^- (\mathcal{G}_{\Lambda})$ is shown similarly.
\end{proof}

Theorem \ref{thm:fundamental_identity} provides an extension of \cite[Theorem 3]{balan2006density2} for Gabor frames in $L^2 (\mathbb{R}^d)$ to general coherent frames. The partition technique used in the proof resembles the proof method of \cite[Theorem 5]{balan2006density} (see also \cite{kolountzakis1996structure}), but the above proof crucially avoids the use of a reference system forming a Riesz basis, which is unknown to exist in the setting of the present paper. Instead, the proof compares the given coherent frame to a continuous reproducing formula \eqref{eq:ortho}, much like the density conditions \cite{nitzan2020revisiting, fuehr2017density, mitkovski2020density} for groups with polynomial growth.

\begin{corollary} \label{cor:frame_measure_independent}
 The lower and upper frame measures $M^- (\mathcal{G}_{\Lambda})$ and $M^+ (\mathcal{G}_{\Lambda})$ of a coherent frame $\pi(\Lambda) g$ wth $g \in \mathcal{B}_{\pi}^2$ are independent of the choice of strong F\o lner sequence $(K_n)_{n \in \mathbb{N}}$.
\end{corollary}
\begin{proof}
 By Theorem \ref{thm:fundamental_identity}, it follows that $M^- (\mathcal{G}_{\Lambda}) = d_{\pi} / D^+ (\Lambda)$. Since $D^+ (\Lambda)$ is independent of the choice of a strong F\o lner sequence by \cite[Proposition 5.14]{pogorzelski2021leptin}, the claim for $M^- (\mathcal{G}_{\Lambda})$ follows. The same argument shows the claim for $M^+ (\mathcal{G}_{\Lambda})$.
\end{proof}

\subsection{Density conditions}
Two immediate consequences of Theorem \ref{thm:fundamental_identity} are the following:

\begin{corollary} \label{cor:density_frames}
Let $g \in \mathcal{B}^2_{\pi}$. If $\pi(\Lambda) g$ is a frame for $\Hpi$, then $D^- (\Lambda) \geq d_{\pi}$. If $\pi(\Lambda) g$ is a Riesz basis for $\Hpi$, then $D^+ (\Lambda) = d_{\pi}$.
\end{corollary}
\begin{proof}
If $\pi(\Lambda) g$ is a frame for $\Hpi$ with canonical dual frame $(h_{\lambda})_{\lambda \in \Lambda}$, then
$
0 \leq \langle \pi(\lambda) g, h_{\lambda} \rangle \leq 1$ for all $\lambda \in \Lambda$, so that Theorem \ref{thm:fundamental_identity} yields
$1 \geq M^+ (\mathcal{G}_{\Lambda}) = d_{\pi} / D^- (\Lambda)$. If $\pi(\Lambda) g$ is a Riesz basis, then $(h_{\lambda})_{\lambda \in \Lambda}$ is bi-orthogonal to $\pi (\Lambda) g$, so that $\langle \pi(\lambda) g, h_{\lambda} \rangle = 1$ for all $\lambda \in \Lambda$, and thus $1=M^- (\mathcal{G}_{\Lambda}) = d_{\pi} / D^+ (\Lambda)$ by Theorem \ref{thm:fundamental_identity}.
\end{proof}

Corollary \ref{cor:density_frames} recovers the statement on frames in \cite[Theorem 1.3]{enstad2022dynamical} and \cite[Theorem 3.14]{enstad2022coherent} under a seemingly weaker condition on the generating vector $g \in \Hpi$. Instead of the assumption $V_g g \in W^L (L^2)$, it is assumed in \cite{enstad2022coherent, enstad2022dynamical} that $V_g (\Hpi) \subseteq W^L (L^2)$.

\begin{corollary}
Suppose $\pi(\Lambda) g$ is a frame for $\Hpi$ with $g \in \mathcal{B}^2_{\pi}$ and frame bounds $A, B > 0$. Then
\begin{align} \label{eq:density_framebounds}
    A \leq d_{\pi}^{-1} D^{-} (\Lambda) \| g \|_{\Hpi}^2 \leq d_{\pi}^{-1} D^{+} (\Lambda) \| g \|_{\Hpi}^2 \leq B.
\end{align}
In particular, if $A=B$, then $D^-(\Lambda) = D^+ (\Lambda)$.
\end{corollary}
\begin{proof}
If $\pi(\Lambda) g$ is a frame for $\Hpi$ with canonical dual frame $(h_{\lambda})_{\lambda \in \Lambda}$, then
\[
\langle \pi(\lambda) g, h_{\lambda} \rangle = \langle \pi(\lambda) g, \frameop^{-1} \pi(\lambda) g \rangle \leq \frac{1}{A} \| \pi(\lambda) g \|^2_{\Hpi} = \frac{1}{A} \| g \|^2_{\Hpi}, \quad \lambda \in \Lambda.
\]
Hence, applying Theorem \ref{thm:fundamental_identity} yields $d_{\pi} / D^- (\Lambda) = M^+ (\mathcal{G}_{\Lambda}) \leq A^{-1} \| g \|_{\Hpi}^2$, and thus
\[
A \leq d_{\pi}^{-1} D^- (\Lambda) \| g \|_{\Hpi}^2.
\]
Using instead the lower bound $\langle \pi(\lambda) g, h_{\lambda} \rangle \geq B^{-1} \| g \|_{\Hpi}^2$, it follows by similar arguments that
$
d_{\pi}^{-1} D^+ (\Lambda) \| g \|_{\Hpi}^2 \leq B
$, as required.
\end{proof}

\section{Overcompleteness of coherent frames} \label{sec:overcompleteness}
This section concerns rigidity theorems for coherent frames showing that infinite sets can be removed yet leave a frame.

\subsection{Infinite excess}
The \emph{excess} of a coherent frame $\pi(\Lambda) g$ for $\Hpi$ is the supremum over the cardinalities of all subsets $\Gamma \subseteq \Lambda$ such that $(\pi(\lambda) g)_{\lambda \in \Lambda \setminus \Gamma}$ is complete in $\Hpi$.

Theorem \ref{thm:infinite_removal} shows that overcomplete coherent frames $\pi(\Lambda) g$ with $g \in \mathcal{B}^2_{\pi}$ have infinite excess. For this, the following characterization will be used, cf. \cite[Corollary 5.7]{balan2003deficits}.

\begin{theorem}[\cite{balan2003deficits}]
\label{thm:infinite_excess}
Let $(g_{\lambda})_{\lambda \in \Lambda}$ be a frame for a Hilbert space $\mathcal{H}$ with canonical dual frame $(h_{\lambda} )_{\lambda \in \Lambda}$. Then the following are equivalent:

\begin{enumerate}
    \item[(i)] There exists an infinite subset $\Gamma \subseteq \Lambda$ such that $(g_{\lambda})_{\lambda \in \Lambda \setminus \Gamma}$ is a frame for $\mathcal{H}$;
    \item[(ii)] There exists $\alpha \in (0,1)$ and an infinite subset $\Lambda_{\alpha} \subseteq \Lambda$ such that $\sup_{\lambda \in \Lambda_{\alpha}} \langle g_{\lambda}, h_{\lambda} \rangle \leq \alpha$.
\end{enumerate}
\end{theorem}

\begin{theorem} \label{thm:infinite_removal}
Suppose $\pi(\Lambda) g$ is a frame for $\Hpi$ with $g \in \mathcal{B}^2_{\pi}$ and $D^+ (\Lambda) > d_{\pi}$.

There exists an infinite set $\Gamma \subseteq \Lambda$ such that $(\pi(\lambda) g )_{\lambda \in \Lambda \setminus \Gamma}$ is a frame for $\Hpi$.
\end{theorem}
\begin{proof}
An application of Theorem \ref{thm:fundamental_identity} yields that $M^- (\mathcal{G}_{\Lambda}) = d_{\pi} / D^+ (\Lambda) < 1$, where the inequality follows by assumption. Therefore, there exists $\varepsilon > 0$ and sequences $(x_i)_{i \in \mathbb{N}}$ and $(n_i)_{i \in \mathbb{N}}$ in $G$ resp. $\mathbb{N}$ such that
\[
 \frac{1}{\# (\Lambda \cap x_i K_{n_i})} \sum_{\lambda \in \Lambda \cap x_i K_{n_i}} \langle \pi(\lambda) g , h_{\lambda} \rangle < 1 - 2\varepsilon
\]
for all $i \in \mathbb{N}$. Since $0 < \langle \pi(\lambda) g, h_{\lambda} \rangle \leq 1$, it follows that at least $\varepsilon / (1-\varepsilon) \cdot \# (\Lambda \cap x_i K_{n_i})$ of the terms $\langle \pi(\lambda)g, h_{\lambda} \rangle$, where $\lambda \in \Lambda \cap x_i K_{n_i}$, satisfy $\langle \pi(\lambda)g, h_{\lambda} \rangle \leq 1 - \varepsilon$. Therefore, there exists an infinite set $\Lambda' \subseteq \Lambda$ such that $\sup_{\lambda \in \Lambda'} \langle \pi(\lambda) g, h_{\lambda} \rangle \leq 1 - \varepsilon$. Hence, the conclusion follow by Theorem \ref{thm:infinite_excess}.
\end{proof}

 Theorem \ref{thm:infinite_removal} can also be deduced from a combination of Theorem \ref{thm:fundamental_identity} and the relation between excess and the ultrafilter frame measure function defined in \cite{balan2007measure};  see \cite[Theorem 4.4]{balan2007measure}.

\subsection{Positive density removal}
This section provides two results on the removal of sets with positive density, which is a stronger conclusion than the removal of merely infinite sets provided by Theorem \ref{thm:infinite_removal}.
The first result is the following necessary condition, which is an adaption of \cite[Proposition 2]{balan2006density} to the setting of the present paper.

\begin{proposition} \label{prop:positive_density_removal}
Suppose that $\pi(\Lambda) g$ is a frame for $\Hpi$ with $g \in \mathcal{B}^2_{\pi}$. If there exists a subset $\Gamma \subseteq \Lambda$ with density $D^- (\Gamma) > 0$ such that $(\pi(\lambda) g)_{\lambda \in \Lambda \setminus \Gamma}$ is a frame for $\Hpi$, then $D^+ (\Lambda) > d_{\pi}$.
\end{proposition}
\begin{proof}
Let $(h_{\lambda})_{\lambda \in \Lambda}$ be the canonical dual frame of $\pi(\Lambda) g$.
Suppose $\Gamma \subseteq \Lambda$ is as in the statement and that $(\pi(\lambda) g)_{\lambda \in \Lambda \setminus \Gamma}$ is a frame for $\Hpi$. Then also $(\frameop^{-1/2} \pi(\lambda) g)_{\lambda \in \Lambda \setminus \Gamma}$ is a frame for $\Hpi$ with lower frame bound, say, $A > 0$. Since $\frameop^{-1/2} \pi(\Lambda) g$ is a Parseval frame for $\Hpi$, given $\gamma \in \Gamma$, the optimal lower frame bound $A'_{\gamma} > 0$ of the frame $(\frameop^{-1/2} \pi (\lambda) g)_{\lambda \in \Lambda \setminus \{\gamma\}}$ is
\[ A'_{\gamma} = 1 - \|\frameop^{-1/2} \pi(\gamma) g \|_{\Hpi}^2 = 1 - \langle \pi(\gamma) g, \frameop^{-1} \pi(\gamma) g \rangle. \]
Therefore, it necessarily follows that $A \leq A_{\gamma}' = 1  - \langle \pi(\gamma) g, h_{\gamma} \rangle$ for all $\gamma \in \Gamma$, which implies that $\Gamma \subseteq \Lambda' := \{ \lambda \in \Lambda : \langle \pi(\lambda) g, h_{\lambda} \rangle \leq 1 - A \}$. Thus, $D^- (\Lambda') > 0$.

For showing that $D^+ (\Lambda) > d_{\pi}$, it now suffices to show the upper bound
\begin{align} \label{eq:lambda-A}
 D^- (\Lambda') \leq  \frac{1}{A} D^+(\Lambda) ( 1 - d_{\pi}/D^+ (\Lambda)).
\end{align}
The inequality \eqref{eq:lambda-A} is trivially satisfied whenever $d_{\pi} / D^+ (\Lambda) \leq 1 - A$. Assume therefore that $1 \geq d_{\pi} / D^+ (\Lambda) > 1 - A$.
We have for any $x \in G$ and compact $K \subseteq G$ such that $\Lambda \cap x K$ is nonempty,
\begin{align*}
& \frac{1}{\# (\Lambda \cap x K )} \sum_{\lambda \in \Lambda \cap x K} \langle \pi(\lambda) g, h_{\lambda} \rangle \\
    &  \leq  \frac{1}{\# (\Lambda \cap x K )}
    \bigg( \sum_{\lambda \in \Lambda' \cap x  K } \langle \pi(\lambda) g, h_{\lambda} \rangle + \sum_{\lambda \in \Lambda \setminus \Lambda' \cap x  K } \langle \pi(\lambda) g, h_{\lambda} \rangle \bigg) \\
    &\leq \frac{(1-A) \cdot  \# (\Lambda' \cap x  K ) + \# (\Lambda \setminus \Lambda' \cap x K )}{\# (\Lambda \cap x K )} \\
    &= \frac{ \# (\Lambda \cap x K ) - A \cdot  \# (\Lambda' \cap x K )}{\# (\Lambda \cap x  K )},
\end{align*}
which yields that
\begin{align} \label{Eqn=LastThing}
 \frac{\#( \Lambda' \cap x K )}{\mu_G (K) } \leq \frac{1}{A}  \frac{\#( \Lambda \cap x K )}{\mu_G (K) } \left(1 - \frac{1}{\# (\Lambda \cap x K )} \sum_{\lambda \in \Lambda \cap x K} \langle \pi(\lambda) g, h_{\lambda} \rangle \right).
\end{align}
Let $\varepsilon > 0$ be arbitrary. Choose a sequence of  $x_i \in G$  and  increasing $n_i  \in \mathbb{N}$ such that $\Lambda \cap x_i K_{n_i}$ is nonempty and
\begin{align*}
\left|  \frac{1}{\# (\Lambda \cap x_i K_{n_i} )} \sum_{\lambda \in \Lambda \cap x_i K_{n_i} } \langle \pi(\lambda) g, h_{\lambda} \rangle - M^- (\mathcal{G}_{\Lambda}) \right| < \varepsilon.
\end{align*}
There exists $j = j (\varepsilon) \in \mathbb{N}$ such that, for all $i \geq j$,
\[
D^- (\Lambda') - \varepsilon  \leq  \frac{\#( \Lambda' \cap x_i K_{n_i} )}{\mu_G (K_{n_i}) }.
\]
Combining this with the inequality \eqref{Eqn=LastThing} yields that
\begin{align}\label{Eqn=Fix}
 D^- (\Lambda') - \varepsilon \leq \frac{1}{A}  \frac{\#( \Lambda \cap x_i K_{n_i} )}{\mu_G (K_{n_i}) } (1 -  M^- (\mathcal{G}_{\Lambda})  + \varepsilon)
\end{align}
for all $i \geq j$. Therefore, by  Theorem \ref{thm:fundamental_identity},
\begin{align*}
 D^- (\Lambda') - \varepsilon  \leq \frac{1}{A}  D^+(\Lambda) \left(1 - M^- (\mathcal{G}_{\Lambda})  +\epsilon  \right) = \frac{1}{A}  D^+(\Lambda)  (1 - d_\pi/D^+(\Lambda) + \varepsilon  ).
\end{align*}
As $\varepsilon > 0$ was chosen arbitrary, this shows \eqref{eq:lambda-A} and finishes the proof.
\end{proof}

The last result shows that for a coherent frame $\pi(\Lambda) g$ with $g \in \mathcal{B}_{\pi}^1$ one can always remove a set of positive density yet leave a frame. For this, the following simple lemma will be used, cf. \cite[Lemma 5]{balan2006density}.

\begin{lemma}[\cite{balan2006density}] \label{lem:positive_density_removal}
Let $(g_{\lambda})_{\lambda \in \Lambda}$ be a frame for $\mathcal{H}$ with frame operator $S : \mathcal{H} \to \mathcal{H}$. For $\Gamma \subseteq \Lambda$, define the truncated coefficient operator $\coeffGam : \Hpi \to \ell^2 (\Gamma)$ by $\coeffGam  = (\langle \cdot, g_{\gamma}\rangle)_{\gamma \in \Gamma}$.
Then $(g_{\lambda})_{\lambda \in \Lambda \setminus \Gamma}$ is a frame for $\mathcal{H}$ if and only if $ \big\|\coeffGam \frameop^{-1} \coeffGam^* \big\|_{B(\ell^2)} < 1$.
\end{lemma}

\begin{theorem}
Suppose $\pi(\Lambda) g$ is a frame for $\Hpi$ with $g \in \mathcal{B}^1_{\pi}$ and $D^- (\Lambda) > d_{\pi}$. Then there exists $\Gamma \subseteq \Lambda$ such that $D^- (\Gamma) > 0$ and $(\pi(\lambda) g)_{\lambda \in \Lambda \setminus \Gamma}$ is a frame for $\Hpi$.
\end{theorem}
\begin{proof}
By re-scaling $\pi(\Lambda) g$ if necessary, it may be assumed that $\pi(\Lambda) g$ is a frame with frame bounds $0 < A < B <2$.
Since $g \in \mathcal{B}^1_{\pi} \subseteq \mathcal{B}^2_{\pi}$ (cf. Lemma \ref{lem:bettervectors}) and $D^- (\Lambda) > d_{\pi}$, it follows by Theorem \ref{thm:fundamental_identity} that $M^+ (\mathcal{G}_{\Lambda}) = d_{\pi} / D^- (\Lambda) < 1$. Fix $\alpha \in (0,1)$ such that $M^+ (\mathcal{G}_{\Lambda}) < \alpha < 1$.
\\~\\
\textbf{Step 1.} In this step, it will be shown that the set $\Lambda_{\alpha} := \{ \lambda \in \Lambda : \langle \pi(\lambda) g, h_{\lambda} \rangle < \alpha \}$ has positive lower Beurling density.
It follows from the definition of $\Lambda_{\alpha}$ that for $x \in G$ and   compact  $K \subseteq G$ such that $\Lambda \cap x K$ is non-empty,
\begin{align*}
  & \frac{1}{\# (\Lambda \cap x K)} \sum_{\lambda \in \Lambda \cap x K} \langle \pi(\lambda) g, h_{\lambda} \rangle \\
    &= \frac{1}{\# (\Lambda \cap x K)} \bigg( \sum_{\lambda \in \Lambda_{\alpha} \cap x K} \langle \pi(\lambda) g, h_{\lambda} \rangle + \sum_{\lambda \in \Lambda \setminus \Lambda_{\alpha} \cap x K} \langle \pi(\lambda) g, h_{\lambda} \rangle \bigg) \\
    &\geq \frac{1}{\# (\Lambda \cap x K)} \bigg( \sum_{\lambda \in \Lambda_{\alpha} \cap x K} 0 + \sum_{\lambda \in \Lambda \setminus \Lambda_{\alpha} \cap x K} \alpha  \bigg) \\
    &= \alpha \frac{\# (\Lambda \cap x K ) - \# (\Lambda_{\alpha} \cap x K)}{\# (\Lambda \cap x K )}.
\end{align*}
Hence,
\begin{align*}
  \frac{\# (\Lambda_{\alpha} \cap x K )}{\mu_G (K)} \geq \bigg( 1 -  \alpha^{-1} \frac{1}{ \# (\Lambda \cap x K)} \sum_{\lambda \in \Lambda \cap x K} \langle \pi(\lambda) g, h_{\lambda} \rangle  \bigg) \frac{\# (\Lambda \cap x K  )}{\mu_G (K)}.
\end{align*}
Let $\varepsilon > 0$ be arbitrary. Take a sequence of $x_i \in G$ and increasing $n_i \in \mathbb{N}$ with  $\Lambda \cap x_i K_{n_i}$   nonempty such that
\[
\left| \frac{1}{  \# (\Lambda \cap x_i K_{n_i})} \sum_{\lambda \in \Lambda \cap x_i K_{n_i}} \langle \pi(\lambda) g, h_{\lambda} \rangle - M^+ (\mathcal{G}_{\Lambda})  \right| < \varepsilon.
\]
Choose $i$ sufficiently large such that
\[
D^- (\Lambda) - \varepsilon \leq  \frac{\# (\Lambda \cap x_i K_{n_i}  )}{\mu_G (K_{n_i})}.
\]
Then we find that
\[
  \frac{\# (\Lambda_{\alpha} \cap x_i K_{n_i} )}{\mu_G (K_{n_i})} \geq    \big( 1 -  \alpha^{-1} (  M^+ (\mathcal{G}_{\Lambda}) - \varepsilon )   \big)  \frac{\# (\Lambda \cap x_i K_{n_i}  )}{\mu_G (K_{n_i})} \geq  \big( 1 -  \alpha^{-1} ( M^+ (\mathcal{G}_{\Lambda}) - \varepsilon) \big)  (D^- (\Lambda) - \varepsilon).
\]
As by assumption $M^+ (\mathcal{G}_{\Lambda}) < \alpha$, $D^- (\Lambda) > d_\pi$ and $\varepsilon >0$ may be chosen arbitrary chosen arbitrarily small, this shows that $D^- (\Lambda_{\alpha}) > 0$.
\\~\\
\textbf{Step 2.} This step provides a convenient expression for $\coeff \frameop^{-1} \coeff^*$ to apply Lemma \ref{lem:positive_density_removal}.
 For this, recall that the frame operator $\frameop$ is positive with $0 < A \leq \frameop \leq B < 2$, so that $\| I - \frameop \|_{B(\Hpi)} < 1$. Therefore, $\frameop^{-1}$ can be expanded as
\[
\frameop^{-1} = \sum_{j = 0}^{\infty} (I - \frameop)^j = \sum_{j = 0}^{\infty} (I - \coeff^* \coeff)^j.
\]
Since $\coeff (I - \coeff^* \coeff) = (I - \coeff \coeff^*) \coeff$, it follows by induction that
\[ \coeff (I - \coeff^* \coeff)^j = (I - \coeff \coeff^*)^j \coeff, \quad j \in \mathbb{N}, \]
and hence
\[
\coeff \frameop^{-1} \coeff^* = \coeff \sum_{j = 0}^{\infty} (I - \coeff^* \coeff)^j \coeff^* =  \sum_{j = 0}^{\infty} (I - \coeff \coeff^*)^j \coeff \coeff^*
\]
with convergence in the operator norm.  For $N \in \mathbb{N} \cup \{\infty\}$, define $M^{(N)} \in \mathbb{C}^{\Lambda \times \Lambda}$ by
\[ M^{(N)} := \sum_{j = 0}^N (I - \coeff \coeff^*)^j \coeff \coeff^*\]
and write $M^{(N)} = D^{(N)} + R^{(N)}$, where $D^{(N)}$ is the diagonal part of $M^{(N)}$. Note that, in particular, we have that $M^{(\infty)} = \coeff \frameop^{-1} \coeff^*$.
\\~\\
\textbf{Step 3.}
This step will show the existence of a subset $\Gamma \subseteq \Lambda_{\alpha}$ such that $D^- (\Gamma) > 0$ and $\| \coeffGam \frameop^{-1} \coeffGam^* \|_{B(\ell^2)} < 1$. It follows then by Lemma \ref{lem:positive_density_removal} that $(\pi(\lambda) g)_{\lambda \in \Lambda \setminus \Gamma}$ is a frame for $\Hpi$.
Throughout this step, fix $0 < \varepsilon < (1-\alpha)/3$ and choose $N \geq 1$ such that
\begin{align} \label{eq:MN-eps}
\big\| \coeff \frameop^{-1} \coeff^* - M^{(N)} \big\|_{B(\ell^2)}  = \big\| M^{(\infty)} - M^{(N)} \big\|_{B(\ell^2)}  \leq \varepsilon.
\end{align}

Since $g \in \mathcal{B}^1_{\pi}$ by assumption, it follows that the matrix $(\langle \pi(\lambda) g, \pi(\lambda') g \rangle )_{\lambda, \lambda' \in \Lambda}$ associated to the operator $\coeff \coeff^* : \ell^2 (\Lambda) \to \ell^2 (\Lambda)$ satisfies
\[ |\langle \pi(\lambda)g, \pi(\lambda') g \rangle | = \Phi (\lambda^{-1} \lambda') = \Phi ((\lambda')^{-1} \lambda), \quad \lambda, \lambda' \in \Lambda, \]
for $\Phi := |V_g g| \in W(G)$. The matrix $M^{(N)}$ being a sum of products involving $(\langle g_{\lambda}, g_{\lambda'} \rangle )_{\lambda, \lambda' \in \Lambda}$ and $I$, it follows therefore
by \cite[Proposition 4.6]{romero2021dual} that there exists $\Theta \in W(G)$ such that
\[
|M^{(N)}_{\lambda, \lambda'} | \leq \min \{ \Theta ((\lambda')^{-1} \lambda), \Theta(\lambda^{-1} \lambda') \}, \quad \lambda, \lambda' \in \Lambda.
\]
Choose a compact symmetric unit neighborhood $U_1 \subseteq G$ such that
\begin{align} \label{eq:amalgam-eps}
\| \Theta \cdot \mathds{1}_{U_1^c} \|_{W} \leq \varepsilon \cdot \bigg(\frac{\Rel(\Lambda)}{\mu_G (Q)} \bigg)^{-1},
\end{align}
where $Q \subseteq G$ is the fixed compact symmetric unit neighborhood.
On the other hand, since $D^- (\Lambda_{\alpha} ) > 0$, there also exists a compact symmetric unit neighborhood $U_2 \subseteq G$ such that $G = \bigcup_{\lambda \in \Lambda_{\alpha}} \lambda U_2$. Set $U := U_2 U_1$ and let $\Gamma \subseteq \Lambda_{\alpha}$ be a maximal family such that $(\gamma U_1)_{\gamma \in \Gamma}$ consists of pairwise disjoint sets.
For showing that $D^- (\Gamma) > 0$, it suffices to show that $\Gamma$ is  relatively dense, see, e.g., \cite[Lemma 3.8]{pogorzelski2021leptin}. Arguing by contradiction,
assume that there exists $x \in G$ such that $\Gamma \cap x U = \emptyset$. Since $\Lambda_{\alpha}$ is $U_2$-dense, there exists $\lambda_0 \in \Lambda_{\alpha} \cap x U_2$. Note that $\lambda_0 \notin \Gamma$. Set $\Gamma_0 := \Gamma \cup \{\lambda_0\}$.
By maximality of $\Gamma$, the family $(\gamma U_1)_{\gamma \in \Gamma_0}$ is not pairwise disjoint, so that there exists
$
 \gamma_0 \in \lambda_0 U_1 \cap  \Gamma_0 \setminus \{\lambda_0\}$.
Since $\gamma_0 \in \lambda_0 U_1$ and $\lambda_0 \in x U_2$, it follows that
\[ \gamma_0 \in \Gamma \cap x U_2 U_1 = \Gamma \cap x U, \]
which contradicts that $\Gamma \cap x U = \emptyset$. Thus, $\Gamma$ is $U$-dense, and $D^- (\Gamma) > 0$.

It remains to show that $\| \coeffGam \frameop^{-1} \coeffGam^* \|_{B(\ell^2)} < 1$. For this, note first that
\begin{align*}
    \| \coeffGam \frameop^{-1} \coeffGam^* \|_{B(\ell^2)}
    &= \| P_{\Gamma} \coeff \frameop^{-1} \coeff^* P_{\Gamma} \|_{B(\ell^2)} ,
    \end{align*}
where $P_{\Gamma} : \ell^2 (\Lambda) \to \ell^2 (\Lambda)$ is the projection operator given by $(P_{\Gamma} c)_{\gamma} = c_{\gamma}$ for $\gamma \in \Gamma$, and $0$ otherwise. Using the notation from Step 2, this yields
    \begin{align*}
    \| \coeffGam \frameop^{-1} \coeffGam^* \|_{B(\ell^2)}
     &\leq \| P_{\Gamma} R^{(N)} P_{\Gamma} \|_{B(\ell^2)}
    + \| P_{\Gamma} D^{(N)} P_{\Gamma} \|_{B(\ell^2)}
    + \| P_{\Gamma} (M^{(\infty)} - M^{(N)} ) P_{\Gamma} \|_{B(\ell^2)}  \\
    &\leq \| P_{\Gamma} R^{(N)} P_{\Gamma} \|_{B(\ell^2)}   + \| P_{\Gamma} D^{(\infty)} P_{\Gamma} \|_{B(\ell^2)} \\
 & \quad \quad    + \| P_{\Gamma} (D^{(\infty)} - D^{(N)}) P_{\Gamma} \|_{B(\ell^2)}  + \| P_{\Gamma} (M^{(\infty)} - M^{(N)} ) P_{\Gamma} \|_{B(\ell^2)}.
    \end{align*}
By Equation \eqref{eq:MN-eps}, it follows that
$\| P_{\Gamma} (D^{(\infty)} - D^{(N)}) P_{\Gamma} \|_{B(\ell^2)} \leq \varepsilon$, which also implies that
$\| P_{\Gamma} (M^{(\infty)} - M^{(N)} ) P_{\Gamma} \|_{B(\ell^2)} \leq \varepsilon$.
In addition, since $M^{(\infty)} = (\langle g_{\lambda}, h_{\lambda'} \rangle)_{\lambda, \lambda' \in \Lambda}$, it follows by definition of $\Lambda_{\alpha}$ that
\[ \| P_{\Gamma} D^{(\infty)} P_{\Gamma} \|_{B(\ell^2)} \leq \sup_{\gamma \in \Gamma} \; \langle g_{\gamma}, h_{\gamma} \rangle \leq \alpha. \]
Lastly, consider the matrix $(R^{(N)}_{\gamma, \gamma'})_{\gamma, \gamma' \in \Gamma}$. For $\gamma, \gamma' \in \Gamma$ with $\gamma \neq \gamma'$, it follows that $(\gamma')^{-1} \gamma \notin U_1$ since the family $(\gamma U_1)_{\gamma \in \Gamma}$ is pairwise disjoint by construction of $\Gamma$. Thus,
\[
|R^{(N)}_{\gamma, \gamma'}| \leq \min \{ \Theta ((\gamma')^{-1} \gamma), \Theta(\gamma^{-1} \gamma') \}, \quad \gamma \neq \gamma', \; \gamma, \gamma' \in \Gamma.
\]
On the other hand, $|R^{(N)}_{\gamma, \gamma}| = 0$ by definition. Therefore, setting $\Theta' := \Theta \cdot \mathds{1}_{U_1^c}$ yields
\[
|R^{(N)}_{\gamma, \gamma'}| \leq \min \{ \Theta' ((\gamma')^{-1} \gamma), \Theta' (\gamma^{-1} \gamma') \}, \quad  \gamma, \gamma' \in \Gamma.
\]
Applying \cite[Proposition 4.6]{romero2021dual} therefore gives
\[
\| P_{\Gamma} R^{(N)} P_{\Gamma} \|_{B(\ell^2)} \leq \frac{\Rel (\Gamma)}{\mu_G (Q)} \| \Theta' \|_W \leq \frac{\Rel (\Lambda)}{\mu_G (Q)} \| \Theta \cdot \mathds{1}_{U_1^c} \|_W \leq \varepsilon,
\]
where the last inequality follows by Equation \eqref{eq:amalgam-eps}.
In conclusion, a combination of the estimates above gives
$
    \| \coeffGam \frameop^{-1} \coeffGam^* \|_{B(\ell^2)}
    \leq \varepsilon + \alpha + \varepsilon + \varepsilon < 1,
$
which completes the proof.
\end{proof}

Theorem \ref{thm:infinite_removal} recovers \cite[Theorem 6]{balan2006density2} in the case of Gabor systems. In contrast to the proof of \cite[Theorem 6]{balan2006density2} (see \cite[Theorem 8]{balan2006density}), the proof provided above does not use techniques relying on spectral invariance, which are only available in settings with polynomial growth \cite{sun2007wiener}. The possibility of providing a proof without these techniques was mentioned in \cite[p.133]{balan2006density}.

\section*{Acknowledgements}
M.C. is supported by the NWO Vidi grant ‘Non-commutative harmonic analysis and rigidity of operator algebras’, VI.Vidi.192.018. J.v.V. gratefully acknowledges support from the Austrian Science Fund (FWF) project J-4445.

\end{document}